\documentclass[twoside, 12pt]{article}
\usepackage{amsmath,amssymb,amsthm,mathrsfs}
\usepackage[margin=2.5cm]{geometry} 
\usepackage{lipsum}
\usepackage{titlesec,hyperref}
\usepackage{fancyhdr}
\usepackage[numbers,sort&compress]{natbib}
\usepackage{color}

\fancypagestyle{plain}
{
\fancyhf{}

}

\titleformat{\subsection}{\it}{\thesubsection.\enspace}{1.5pt}{}
\titleformat{\subsubsection}{\it}{\thesubsubsection.\enspace}{1.5pt}{}

\newtheorem{theo}{Theorem}[section]
\newtheorem{lemm}[theo]{Lemma}

\newtheorem{rema}{Remark}[section]
\numberwithin{equation}{section}

\allowdisplaybreaks

\def\th2{\frac{\theta}{2}}

\begin{document}
\title{ Asymptotic Behavior of Solution to the Incompressible Nematic Liquid Crystal Flows in $\mathbb{R}^3$  \hspace{-4mm}}
\author{Jincheng Gao$^\dag$  \quad Qiang Tao $^\ddag$ \quad Zheng-an Yao $^\dag$\\[10pt]
\small {$^\dag $School of Mathematics and Computational Science, Sun Yat-Sen University,}\\
\small {510275, Guangzhou, P. R. China}\\[5pt]
\small {$^\ddag$ College of Mathematics and Computational Science, Shenzhen University,}\\
\small {518060, Shenzhen, P. R. China}\\[5pt]
}

\footnotetext{Email: \it gaojc1998@163.com(J.C.Gao), taoq060@126.com(Q.Tao), \it mcsyao@mail.sysu.edu.cn(Z.A.Yao).}
\date{}
\maketitle

\begin{abstract}
In this paper, we investigate the Cauchy problem for the incompressible nematic liquid crystal flows in three-dimensional whole space. First of all, we establish the global existence of solution by energy method under assumption of small initial data. Furthermore, the time decay rates of velocity and director are built when the initial data belongs to $L^1(\mathbb{R}^3)$ additionally. Finally, one also constructs the time convergence rates for the mixed space-time derivatives of velocity and director.

\vspace*{5pt}
\noindent{\it {\rm Keywords}}: incompressible nematic liquid crystal flows; global solution; asymptotic behavior;
                         Fourier splitting method
\vspace*{5pt}

\noindent{\it {\rm 2010 Mathematics Subject Classification:}}\ {\rm 35Q35, 35B40, 76A15}
\end{abstract}

\section{Introduction}
\quad In this paper, we investigate the motion of incompressible nematic liquid crystal flows,
which are governed by the following simplified version of the Ericksen-Leslie equations
\begin{equation}\label{1.1}
\left\{
\begin{aligned}
& u_t+ u \cdot \nabla u-\mu \Delta u+\nabla P=- {\rm div}(\nabla d \odot \nabla d),\\
& {\rm div}u=0,\\
&d_t+u\cdot \nabla d=\Delta d+|\nabla d|^2 d,
\end{aligned}
\right.
\end{equation}
where $u, p$ and $d$ stand for the velocity, pressure and macroscopic average of the nematic
liquid crystal orientation field respectively. The positive constant $\mu$
represents shear viscosity coefficient of the fluid. For the sake of simplicity, we normalize the constant $\mu$
to be $1$. The term $\nabla d \odot \nabla d$ in the equations of conservation of momentum denotes
the $3\times 3$ matrix whose $ij-$th entry is given by $\partial_i d \cdot \partial_j d \ (1 \le i, j \le 3)$.
To complete the system \eqref{1.1}, the initial data is given by
\begin{equation}\label{1.2}
\left.(u,d)(x,t)\right|_{t=0}=(u_0(x), d_0(x)).
\end{equation}
The initial data is also supposed to satisfy ${\rm div}u_0(x)=0$ and $|d_0(x)|=1.$
As the space variable tends to infinity, we suppose
\begin{equation}\label{1.3}
\underset{|x|\rightarrow \infty}{\lim}(u_0,d_0-w_0)(x)=0,
\end{equation}
where $w_0$ is a fixed constant unit vector.
The system \eqref{1.1} is a coupling between the incompressible homogeneous Navier-Stokes equations and a transported
heat flow of harmonic maps into $S^2$. It is a macroscopic continuum description of the evolution for
the liquid crystals of nematic type under the influence of both the flow field $u$ and the macroscopic
description of the microscopic configuration $d$ of rod-like liquid crystals. Generally speaking,
the system \eqref{1.1} can not be obtained any better results than Navier-Stokes equations.

The hydrodynamic theory of liquid crystals in the nematic case has been established by Ericksen \cite{Ericksen}
and Leslie \cite{Leslie} during the period of $1958$ through $1968$.
Since then, the mathematical theory is still progressing and the study of the full Ericksen-Leslie model
presents relevant mathematical difficulties. The pioneering work comes from Lin and his partners
\cite{{Hard-Kinderlehrer-Lin}, {Lin},{Lin-Liu1}, {Lin-Liu2}}.
For example, Lin and Liu \cite{Lin-Liu1} obtained
the global weak and smooth solutions for the Ginzburg-Landau approximation to relax the nonlinear constraint
$d \in S^2.$ They also discussed the uniqueness and some stability properties of the system.
Later, a series of papers concerning about the time decay rates for this approximated system in a bounded domain
were given by Wu \cite{Wuhao}, Hu and Wu \cite{Hu-Wu2}, Grasselli and Wu \cite{Grasselli-Wu} respectively.
On the other hand,  Dai, Qing and Schonbek \cite{{Dai1}}, Dai and Schonbek \cite{{Dai2}} established the time decay rates
for the Cauchy problem respectively. More precisely, Dai and Schonbek \cite{{Dai2}} obtained the global existence of solution
in the Sobolev space $H^N(\mathbb{R}^3)\times H^{N+1}(\mathbb{R}^3)(N \ge 1)$ only requiring the smallness of $\|u_0\|_{H^1}^2+\|d_0-w_0\|_{H^2}^2$, where $w_0$ is an unit constant vector.
They also established the time decay rates
\begin{equation*}
\|\nabla^k u(t)\|_{L^2}+\|\nabla^k (d-w_0)(t)\|_{L^2} \le C(1+t)^{-\frac{3+2k}{4}},
\end{equation*}
for $k=0,1,2...,N$. Recently, Hu and Wang \cite{Hu-Wang} obtained the existence and uniqueness
of the global strong solution with small initial data in a three-dimensional bounded domain with smooth boundary.
They also proved the global weak solutions constructed in \cite{Lin-Liu1} must be equal to the unique strong
solution if the later exists. For the density dependent fluid, Liu and Zhang \cite{Liu-Zhang}
obtained the global weak solutions in dimension three with the initial density $\rho_0 \in L^2$, which was improved
by Jiang and Tan \cite{Jiang-Tan} for the case $\rho_0 \in L^{\gamma}(\gamma > \frac{3}{2})$.

Under the constraint $d\!\in S^2$, for the case of homogeneous fluid, Li and Wang \cite{Li-Wang1} established
the local strong solution with large initial data and the global one with small data.
This result was improved to the case of inhomogeneous fluid in \cite{Li-Wang2}. For the general large initial
data, Huang and Wang \cite{Huang-Wang} established blow up criterion for the short time classical solution of the
nematic liquid crystal flows for the dimension two and three respectively.
Later, Hong \cite{Hong} and Lin, Lin, and Wang \cite{Lin-Lin-Wang} showed independently the global existence of a weak solution in two-dimensional space. Later, Xu and Zhang  \cite{Xu-Zhang} proved the global existence and regularity of weak solutions
for the liquid crystal flows with large initial velocity in two-dimensional whole space.
Recently, Lin and Wang \cite{Lin-Wang} established the global existence of a weak solution for the initial-boundary value or the Cauchy problem by restricting the initial director field on the unit upper hemisphere.
On the other hand, Wang \cite{Wang}
established global well-posedness theory for rough initial data provided that $\|u_0\|_{{\rm BMO}^{-1}}+[d_0]_{BMO} \le \varepsilon_0$ for some $\varepsilon_0 >0.$
The local well-posedness for initial data $(u_0, d_0)$ with small $L^3_{uloc}(\mathbb{R}^3)$-norm has been proved by Hineman and Wang \cite{Hineman-Wang}.
Under smallness of $\|u_0\|_{{\rm BMO}^{-1}}+[d_0]_{BMO}$, Du and Wang \cite{Du-Wang1} obtained arbitrary space-time regularity for the Koch and Tataru type solution $(u, d)$. As a corollary, they also got the decay rates.
For the case of nonhomogeneous fluid, Wen and Ding \cite{Wen-Ding} established the local existence for the
strong solution and obtained the global solution under the assumptions of small energy and positive initial density.
Later, Li \cite{Li1} improved this result to the case of vacuum. At the same time, Li \cite{Li2} established
the local existence and uniqueness of strong solution and extended such local strong solution to be global,
provided the initial density is away from vacuum and the initial direction field satisfies some geometric structure.
For more results, the readers can refer to \cite{Lin-Wang1} that have introduced
some recent developments of analysis for hydrodynamic flow of nematic liquid crystal
flows and references therein.
As for the compressible nematic liquid crystal flows, the readers can refer to
\cite{{Ding-Lin-Wang-Wen},{Ding-Wang-Wen},{Huang-Wang-Wen1},{Huang-Wang-Wen2},{Huang-Wang1},{Ding-Huang-Xia},
{Yang-Dou-Ju},{Jiang-Jiang-Wang},{Ding-Huang-Wen-Zi},{Hu-Wu},{Ma},{Gao-Tao-Yao},{Lin-Lai-Wang}}
and references therein.

Motivated by the work of Dai and Schonbek \cite{{Dai2}}, we are interested in establishing the global solution under assumption of small initial data and studying the asymptotic behavior of solution for the systems \eqref{1.1}-\eqref{1.3}
in this paper. First of all, we construct global solution with the help of the method developed by Guo and Wang \cite{Guo-Wang} only requiring the smallness of $\|u_0\|_{H^1}^2+\|d_0-w_0\|_{H^2}^2$.
Secondly, one obtains the time decay rates for the direction field by the method in \cite{K-N-N}
and the velocity by the Fourier splitting method in \cite{Dai2}. The difficulty here is to deal with the super
nonlinear term $|\nabla d|^2 d$. Furthermore, we establish the time convergence
rates for the higher-order derivatives of velocity and director by taking the strategy of induction guaranteed
by the Fourier splitting method. Motivated by the Lemma \ref{motivated}, we enhance the time decay rates
for the higher-order derivatives of direction field. Finally, the time decay rates for mixed space-time
derivatives of velocity and director are also built.

\textbf{Notation:} In this paper, we use $H^s(\mathbb{R}^3)( s\in \mathbb{R})$ to denote the usual Sobolev space
with norm $\|\cdot\|_{H^s}$ and $L^p(\mathbb{R}^3)(1\le p \le \infty)$ to denote the usual $L^p$ space with norm
$\| \cdot \|_{L^p}$. The symbol $\nabla^l $ with an integer $l \ge 0$ stands for the usual any spatial derivatives
of order $l$. When $l$ is not an integer, $\nabla^l $ stands for $\Lambda^l$ defined by
$\Lambda^l f :=\mathcal{F}^{-1}(|\xi|^l \mathcal{F}f)$, where $\mathcal{F}$ is the usual Fourier transform operator
and $\mathcal{F}^{-1}$ is its inverse.
We will employ the notation $a \lesssim b$ to mean that $a \le C b$ for an universal constant $C>0$ independent of
the time $t$. For the sake of simplicity, we write  $\int f dx =\int _{\mathbb{R}^3} f dx.$

Now, we state our first result concerning the global existence of solution for the incompressible nematic liquid
crystal flows \eqref{1.1}-\eqref{1.3}.

\begin{theo}\label{global-existence}
Assume $(u_0, d_0-w_0) \in H^N \times H^{N+1}$ for any integer $N \ge 1$. If there
exists a constant $\delta_0 >0$ such that
\begin{equation}\label{1.4}
\|u_0\|_{H^1}^2+\|d_0-w_0\|_{H^2}^2 \le \delta_0,
\end{equation}
then the problem \eqref{1.1}-\eqref{1.3} admits a unique global solution $(u, p,d)$ satisfying for all $t \ge 0$,
\begin{equation}\label{1.5}
\|u\|_{H^N}^2+\|d-w_0\|_{H^{N+1}}^2+
 \int_0^t (\|\nabla u\|_{H^N}^2+\|\nabla (d-w_0)\|_{H^{N+1}}^2) d\tau
 \le \|u_0\|_{H^N}^2+\|d_0-w_0\|_{H^{N+1}}^2.
\end{equation}
\end{theo}

After having the unique global solution at hand, we concentrate on discussing the long time behavior of
velocity and direction field. Hence, we establish the decay rates as follow.

\begin{theo}\label{rates1}
Let all the assumptions in Theorem \ref{global-existence} hold on. Suppose additionally $u_0\in L^1(\mathbb{R}^3)$
and $d_0-w_0 \in L^1(\mathbb{R}^3)$, then the global solution $(u, p, d)$ of problem \eqref{1.1}-\eqref{1.3} has
the time decay rates for all $t\ge 0$,
\begin{equation}\label{1.6}
\begin{aligned}
&\|\nabla^k u(t)\|_{L^2} \le C(1+t)^{-\frac{3+2k}{4}},\\
&\|\nabla^m (d-w_0)(t)\|_{L^2} \le C(1+t)^{-\frac{3+2m}{4}},\\
\end{aligned}
\end{equation}
where $k=0, 1, 2,...,N$ and $m=0,1,2,...,N+1$.
\end{theo}

\begin{rema}
In \cite{Schonbek-Wiegner}, Schonbek and Wiegner studied the large time behavior of solutions for the
incompressible Navier-Stokes equations under the $H^N-$framework and established following time decay rates
\begin{equation*}
\|\nabla^k u(t)\|_{L^2} \le C(1+t)^{-\frac{3+2k}{4}},
\end{equation*}
where $k=0, 1, 2,...,N$. Then, it is easy to see that our result \eqref{1.6}
coincides with the result in \cite{Schonbek-Wiegner} for the Navier-Stokes equations.
\end{rema}

\begin{rema}
For the case $N=1$, by virtue of \eqref{1.6} and the Sobolev interpolation inequality, we have
\begin{equation*}
\begin{aligned}
&\|u(t)\|_{L^p}\le C(1+t)^{-\frac{3}{2}\left(1-\frac{1}{p}\right)},\\
&\|(d-w_0)(t)\|_{L^q}\le C(1+t)^{-\frac{3}{2}\left(1-\frac{1}{q}\right)},\\
\end{aligned}
\end{equation*}
for $p\in [2, 6]$ and $q\in [2, \infty]$. Furthermore, if $N \ge 2$, then we obtain
\begin{equation*}
\begin{aligned}
&\|\nabla^k u(t)\|_{L^p}\le C(1+t)^{-\frac{3}{2}\left(1-\frac{1}{p}\right)-\frac{k}{2}},\\
&\|\nabla^m (d-w_0)(t)\|_{L^p}\le C(1+t)^{-\frac{3}{2}\left(1-\frac{1}{p}\right)-\frac{m}{2}},\\
\end{aligned}
\end{equation*}
for $2\le p\le \infty; k=0,1,...,N-2$ and $m=0,1,...,N-1$.
\end{rema}

\begin{rema}
Compared with the time decay rates  in \cite{Dai2}, the excellent feature  of  $\eqref{1.6}_2$ is to provide the time decay rates for the $(N+1)-$th order spatial derivatives of director as
\begin{equation*}
\|\nabla^{N+1} (d-w_0)(t)\|_{L^2} \le C(1+t)^{-\frac{5+2N}{4}},
\end{equation*}
which is motivated and guaranteed by the Lemma \ref{motivated}.
\end{rema}

Finally, we  are devoted to establishing the time convergence rates for the mixed space-time derivatives of
director and velocity.
\begin{theo}\label{rates2}
Let all the assumptions in Theorem \ref{rates1} hold on, then the global solution $(u, p, d)$ of the problem \eqref{1.1}-\eqref{1.3}
has the time decay rates for all $t \ge 0$,
\begin{equation}\label{1.7}
\|\nabla^k d_t(t)\|_{L^2} \le C(1+t)^{-\frac{7+2k}{4}},
\end{equation}
where $k=0,1,2,...,N-1$. Furthermore, if $N \ge 2$, then we also have
\begin{equation}\label{1.8}
\|\nabla^{k+1}P\|_{L^2}+\|\nabla^k u_t(t)\|_{L^2} \le C(1+t)^{-\frac{7+2k}{4}},
\end{equation}
where $k=0,1,2,...,N-2$.
\end{theo}

The paper is organized as follows. In section $2$, we establish the global existence of solution
under the assumption of small initial data by energy method.
In section $3$, one obtains the decay rate for the direction field firstly. Then, the application of
Fourier splitting method helps us get the decay rate for the velocity.
Furthermore, we establish the time decay rates for high-order spatial derivatives of velocity and direction field.
Finally, we also study the convergence rates for the pressure and mixed space-time derivatives
of velocity and director.

\section{Global existence of solution}
\quad In this section, we construct the global existence of solution by energy method under the assumption
of small initial data. Our first section concentrates on establishing some differential inequalities
by assuming that the quantity $\|u(t)\|_{H^1}+\|d(t)-w_0\|_{H^2}$ is small. Then, we close the estimates
by the standard continuity argument and the global existence of solution follows immediately
in subsection $2.2$.

Denoting $n=d-w_0$, then the system \eqref{1.1} becomes
\begin{equation}\label{equations2}
\left\{
\begin{aligned}
& u_t+ u \cdot \nabla u- \Delta u+\nabla P=- {\rm div}(\nabla n \odot \nabla n),\\
& {\rm div}u=0,\\
&n_t+u\cdot \nabla n=\Delta n+|\nabla n|^2 (n+w_0),
\end{aligned}
\right.
\end{equation}
and the initial data is given by
\begin{equation}\label{initial2}
\left.(u,n)(x,t)\right|_{t=0}=(u_0(x), n_0(x)).
\end{equation}
When the space variable tends to infinity, we have
\begin{equation}\label{infinity-decay2}
\underset{|x|\rightarrow \infty}{\lim}(u_0,n_0)(x)=0.
\end{equation}
\begin{rema}
Since the initial data $|d_0(x)|=1$, then it follows immediately $|d(x,t)|=1$ by virtue of the maximum principle.
Hence, we have
\begin{equation}\label{equal1}
|(n+w_0)(x,t)|=|d(x,t)|=1,
\end{equation}
which will be used frequently in the following section.
\end{rema}

\subsection{Energy estimates}
\quad First of all, we recall the classical interpolation inequality that will be used frequently in our context.
\begin{lemm}
{\rm (\cite{Tan-Wang})}Let $2 \le p \le \infty$ and $0 \le m, \alpha \le l$ ; when $p=\infty$ we require further that $m \le \alpha+1$
and $l \ge \alpha+2$. Then we have that for any $f\in C_0^\infty(\mathbb{R}^3)$,
\begin{equation}\label{GN}
\|\nabla^\alpha f\|_{L^p} \lesssim \|\nabla^m f \|_{L^2}^{1-\theta} \|\nabla^l f \|_{L^2}^\theta,
\end{equation}
where $0\le \theta \le 1$ and $\alpha$ satisfy
\begin{equation*}
\alpha+3\left(\frac{1}{2}-\frac{1}{p}\right)=m(1-\theta)+l\theta.
\end{equation*}
\end{lemm}

For sufficiently small $\delta>0$, we assume
\begin{equation}\label{smallness}
\|u(t)\|_{H^1}^2+\|n(t)\|_{H^2}^2\le \delta.
\end{equation}
Then, with the help of interpolation inequality \eqref{GN}, we find
\begin{equation*}
\|\Lambda^s f\|_{L^2}\lesssim \|f\|_{L^2}^{1-\frac{s}{l}}\|\nabla^l f\|_{L^2}^{\frac{s}{l}}
\end{equation*}
for any real number $s\in [0, l]$. Hence, we deduce immediately
\begin{equation*}
\|\Lambda^\alpha u(t)\|_{L^2}+\|\Lambda^\beta n(t)\|_{L^2} \lesssim \delta
\end{equation*}
for any $\alpha \in [0, 1]$ and $\beta \in [0, 2]$.

Now, we establish the first inequality as follows.
\begin{lemm}\label{inequality1}
Under the assumption \eqref{smallness}, then we have
\begin{equation}\label{inequality1}
\frac{d}{dt} \int|n|^2 dx+\int |\nabla n|^2 dx \le 0.
\end{equation}
\end{lemm}
\begin{proof}
Multiplying $\eqref{equations2}_3$ by $n$ and integrating over $\mathbb{R}^3$(by part), it is easy to obtain
\begin{equation}\label{221}
\frac{1}{2}\frac{d}{dt}\int |n|^2 dx+\int |\nabla n|^2 dx=\int |\nabla n|^2 (n+w_0) \cdot n dx.
\end{equation}
The application of Sobolev inequality, together with \eqref{smallness}, yields at once
\begin{equation}\label{222}
\int |\nabla n|^2 (n+w_0) \cdot ndx
\le \int |\nabla n|^2 |n|dx
\lesssim \delta \|\nabla n\|_{L^2}^2.
\end{equation}
Combining \eqref{221} with \eqref{222} yields directly
\begin{equation*}
\frac{1}{2}\frac{d}{dt}\int |n|^2 dx+\int |\nabla n|^2 dx
\lesssim \delta \|\nabla n\|_{L^2}^2,
\end{equation*}
which completes the proof of the lemma.
\end{proof}

Next, we will establish the following estimates that play an important role not only in the
study of global existence of solution but also in the study of the convergence rates
for the higher-order spatial derivatives of velocity and director.
\begin{lemm}\label{inequality2}
Under the assumption \eqref{smallness}, then for any $k=0,1,2,..., N,$ we have
\begin{equation}\label{inequality2}
\frac{d}{dt}\int (|\nabla^k u|^2+|\nabla^{k+1} n|^2) dx
+\int (|\nabla^{k+1} u|^2+|\nabla^{k+2} n|^2) dx\le 0.
\end{equation}
\end{lemm}
\begin{proof}
Taking $k-$th spatial derivatives to $\eqref{equations2}_1$, multiplying the resulting equation by $\nabla^k u$ and integrating over $\mathbb{R}^3$(by part), we obtain
\begin{equation}\label{231}
\frac{1}{2}\frac{d}{dt}\int |\nabla^k u|^2 dx+\int |\nabla^{k+1} u|^2 dx
=-\int \nabla^k \left[u \cdot \nabla u+{\rm div}(\nabla n \odot \nabla n)\right] \cdot \nabla^k u \ dx.
\end{equation}
Differentiating $(k+1)$ times with respect to space variable and multiplying the resulting identity by $\nabla^{k+1} n$,
it is easy to deduce with the help of integration over $\mathbb{R}^3$
\begin{equation}\label{232}
\frac{1}{2}\frac{d}{dt}\int |\nabla^{k+1} n|^2 dx+\int |\nabla^{k+2} n|^2 dx
=\int \nabla^{k+1} \left[-u \cdot \nabla n+|\nabla n|^2(n+w_0)\right] \cdot \nabla^{k+1} n \ dx.
\end{equation}
If adding \eqref{231} to \eqref{232}, then it arrives at
\begin{equation}\label{233}
\begin{aligned}
&\frac{1}{2}\frac{d}{dt}\int (|\nabla^k u|^2+|\nabla^{k+1} n|^2) dx
+\int (|\nabla^{k+1} u|^2+|\nabla^{k+2}n|^2) dx\\
&=-\int \nabla^k (u \cdot \nabla u)\nabla^k u \  dx-\int \nabla^k {\rm div} (\nabla n \odot \nabla n)\nabla^k u \ dx\\
&  \quad -\int \nabla^{k+1} (u \cdot \nabla n) \nabla^{k+1} n \ dx+\int \nabla^{k+1} (|\nabla n|^2 (n+w_0))\nabla^{k+1}n \ dx\\
&=I_1+I_2+I_3+I_4.
\end{aligned}
\end{equation}
For the case $k=0$, then the identity \eqref{233} becomes the form as follows
\begin{equation}\label{234}
\begin{aligned}
&\frac{1}{2}\frac{d}{dt}\int (| u|^2+|\nabla n|^2) dx
+\int (|\nabla u|^2+|\nabla^2 n|^2) dx\\
&=-\int (u \cdot \nabla u) u \  dx-\int  {\rm div} (\nabla n \odot \nabla n) u \ dx\\
&  \ \ -\int \nabla (u \cdot \nabla n) \nabla n \ dx+\int \nabla (|\nabla n|^2 (n+w_0))\nabla n \ dx\\
&=I_1+I_2+I_3+I_4.
\end{aligned}
\end{equation}
In view of the incompressibility $\eqref{equations2}_2$, together with integration by part, it is easy to obtain
\begin{equation}\label{235}
\begin{aligned}
I_1
=\frac{1}{2} \int |u|^2 {\rm div}u\ dx=0.
\end{aligned}
\end{equation}
Integrating by part and applying Holder, Sobolev and Young inequalities, we have
\begin{equation}\label{236}
\begin{aligned}
I_2
&=\int \nabla n \odot \nabla n \cdot \nabla u \ dx\\
&\le \|\nabla n\|_{L^3} \|\nabla n\|_{L^6} \|\nabla u\|_{L^2}\\
&\lesssim \|\nabla n\|_{H^1}  \|\nabla^2 n\|_{L^2} \|\nabla u\|_{L^2}\\
&\lesssim \delta( \|\nabla u\|_{L^2}^2+ \|\nabla^2 n\|_{L^2}^2).
\end{aligned}
\end{equation}
By virtue of integration by part and the incompressibility, along with Holder, Sobolev and Young inequalities,
we conclude immediately
\begin{equation}\label{237}
\begin{aligned}
I_3
&=-\int \nabla u_i \partial_i n\  \nabla n \ dx+\frac{1}{2}\int |\nabla n|^2 {\rm div} u  \ dx\\
&\le \|\nabla n\|_{L^3} \|\nabla n\|_{L^6} \|\nabla u\|_{L^2}\\
&\lesssim \delta (\|\nabla u\|_{L^2}^2+\|\nabla^2 n\|_{L^2}^2).
\end{aligned}
\end{equation}
Exploiting \eqref{smallness}, Holder and Young inequalities, it arrives at
\begin{equation}\label{238}
\begin{aligned}
I_4
&\lesssim \int |\nabla n|^2 |\nabla^2 n| dx+\int |\nabla n|^4dx\\
&\lesssim \|\nabla n\|_{L^3}\|\nabla n\|_{L^6}\|\nabla^2 n\|_{L^2}+\|\nabla n\|_{L^2}\|\nabla^2 n\|_{L^2}^3\\
&\lesssim \|\nabla n\|_{H^1}\|\nabla^2 n\|_{L^2}^2+\|\nabla n\|_{L^2}\|\nabla^2 n\|_{L^2}^3\\
&\lesssim \delta \|\nabla^2 n\|_{L^2}^2.
\end{aligned}
\end{equation}
Substituting \eqref{235}-\eqref{238} into \eqref{234}, we complete the proof of \eqref{inequality2}
for the case $k=0$. Then, it is enough to verify \eqref{inequality2} for the case $1\le k\le N$.
First of all, we deal with the term $I_1$. In fact, applying Leibnitz formula, Holder and Sobolev
inequalities, it arrives at
\begin{equation}\label{239}
\begin{aligned}
I_1 &=-\int\sum_{l=0}^k C_k^l \nabla^l u \nabla^{k+1-l}u\ \nabla^k u dx\\
    &\lesssim \sum_{l=0}^k \|\nabla^l u\|_{L^3} \|\nabla^{k+1-l} u\|_{L^2} \|\nabla^{k+1} u \|_{L^2}.
\end{aligned}
\end{equation}
For the case $0 \le l \le \left[ \frac{k}{2}\right]$, by virtue of the interpolation inequality
\eqref{GN} and a priori estimates \eqref{smallness}, we find
\begin{equation}\label{2310}
\begin{aligned}
&\quad \|\nabla^l u\|_{L^3} \|\nabla^{k+1-l}u\|_{L^2}\|\nabla^{k+1} u\|_{L^2}\\
&\lesssim \|\nabla^\alpha u\|_{L^2}^{1-\frac{l}{k}} \|\nabla^{k+1}u\|_{L^2}^{\frac{l}{k}}
          \|\nabla u\|_{L^2}^{\frac{l}{k}} \|\nabla^{k+1} u\|_{L^2}^{1-\frac{l}{k}}\|\nabla^{k+1} u\|_{L^2}\\
&\lesssim \delta \|\nabla^{k+1}u\|_{L^2}^2,
\end{aligned}
\end{equation}
where $\alpha$ is defined by
\begin{equation}\label{2311}
\alpha=1-\frac{k}{2(k-l)}\in \left[0, \frac{1}{2}\right].
\end{equation}
Similarly, for the case $\left[\frac{k}{2}\right]+1 \le l \le k,$ we have
\begin{equation}\label{2312}
\begin{aligned}
&\quad \|\nabla^l u\|_{L^3} \|\nabla^{k+1-l}u\|_{L^2}\|\nabla^{k+1} u\|_{L^2}\\
&\lesssim \|u\|_{L^2}^{1-\frac{l+\frac{1}{2}}{k+1}} \|\nabla^{k+1}u\|_{L^2}^{\frac{l+\frac{1}{2}}{k+1}}
          \|\nabla^\alpha u\|_{L^2}^{\frac{l+\frac{1}{2}}{k+1}} \|\nabla^{k+1} u\|_{L^2}^{1-\frac{l+\frac{1}{2}}{k+1}}\|\nabla^{k+1} u\|_{L^2}\\
&\lesssim \delta \|\nabla^{k+1} u\|_{L^2}^2,\\
\end{aligned}
\end{equation}
where $\alpha$ is defined by
\begin{equation*}
\alpha=\frac{k+1}{2l+1} \in \left(\frac{1}{2}, 1\right).
\end{equation*}
Substituting \eqref{2310} and \eqref{2312} into \eqref{239}, then we conclude
\begin{equation}\label{2313}
I_1 \lesssim \delta \|\nabla^{k+1} u\|_{L^2}^2.
\end{equation}
By integration by part, Leibnitz formula and Holder inequality, the second term
on the right hand side of \eqref{233} can be estimated as
\begin{equation}\label{2314}
\begin{aligned}
I_2 &=\int \nabla^k (\nabla n \odot \nabla n)\cdot \nabla^{k+1} u \ dx\\
    &=\int\sum_{l=0}^k C_k^l \nabla^{l+1} n \nabla^{k+1-l} n\ \nabla^{k+1} u \ dx\\
    &\lesssim \sum_{l=0}^k \|\nabla^{l+1} n\|_{L^3} \|\nabla^{k+1-l} n\|_{L^6} \|\nabla^{k+1} u \|_{L^2}.
\end{aligned}
\end{equation}
For the case $0 \le l \le \left[ \frac{k}{2}\right]$, in view of \eqref{GN} and Young inequality, it arrives at
\begin{equation}\label{2315}
\begin{aligned}
&\quad \|\nabla^{l+1} n\|_{L^3} \|\nabla^{k+1-l}n\|_{L^6}\|\nabla^{k+1} u\|_{L^2}\\
&\lesssim \|\nabla^\alpha n\|_{L^2}^{1-\frac{l}{k+1}} \|\nabla^{k+2} n\|_{L^2}^{\frac{l}{k+1}}
          \|\nabla n\|_{L^2}^{\frac{l}{k+1}} \|\nabla^{k+2} n\|_{L^2}^{1-\frac{l}{k+1}}\|\nabla^{k+1} u\|_{L^2}\\
&\lesssim \|\nabla^\alpha n\|_{L^2}^{1-\frac{l}{k+1}} \|\nabla^{k+2} n\|_{L^2}^{\frac{l}{k+1}}
          \|\nabla n\|_{L^2}^{\frac{l}{k+1}} \|\nabla^{k+2} n\|_{L^2}^{1-\frac{l}{k+1}}\|\nabla^{k+1} u\|_{L^2}\\
&\lesssim \delta (\|\nabla^{k+1}u\|_{L^2}^2+\|\nabla^{k+2} n\|_{L^2}^2),
\end{aligned}
\end{equation}
where $\alpha$ is defined by
\begin{equation*}
\alpha=1+\frac{k+1}{2(k+1-l)}\in \left[\frac{3}{2}, 2\right].
\end{equation*}
By interpolation inequality \eqref{GN} and Young inequality,
for the case $\left[\frac{k}{2}\right]+1 \le l \le k$, it follows
\begin{equation}\label{2316}
\begin{aligned}
&\quad \|\nabla^{l+1} n\|_{L^3} \|\nabla^{k+1-l} n\|_{L^6}\|\nabla^{k+1} u\|_{L^2}\\
&\lesssim \|\nabla n\|_{L^2}^{1-\frac{l+\frac{1}{2}}{k+1}} \|\nabla^{k+2} n\|_{L^2}^{\frac{l+\frac{1}{2}}{k+1}}
          \|\nabla^\alpha n\|_{L^2}^{\frac{l+\frac{1}{2}}{k+1}} \|\nabla^{k+2} n\|_{L^2}^{1-\frac{l+\frac{1}{2}}{k+1}}
          \|\nabla^{k+1} u\|_{L^2}\\
&\lesssim \delta (\|\nabla^{k+1}u\|_{L^2}^2+\|\nabla^{l+2} n\|_{L^2}^2),
\end{aligned}
\end{equation}
where $\alpha$ is defined by
\begin{equation}\label{2317}
\alpha=1+\frac{k+1}{2l+1} \in \left(\frac{3}{2}, 2\right).
\end{equation}
If inserting \eqref{2315} and \eqref{2316} into \eqref{2314}, we obtain immediately
\begin{equation}\label{2318}
I_2 \lesssim \delta (\|\nabla^{k+1} u\|_{L^2}^2+\|\nabla^{k+2} n\|_{L^2}^2).
\end{equation}
Similar to the estimate $I_1$, we can estimate $I_3$ as follows
\begin{equation}\label{2319}
I_3 \lesssim \delta (\|\nabla^{k+1} u\|_{L^2}^2+\|\nabla^{k+2} n\|_{L^2}^2).
\end{equation}
Since the term $|\nabla n|^2 (n+w_0)$ is a supercritical nonlinear term, the estimate of $I_4$ seems somewhat complicated.
Actually, by integration by part and Leibnitz formula, we find firstly
\begin{equation}\label{2320}
\begin{aligned}
I_4
& =-\int \sum_{l=0}^k C_k^l \nabla^l(|\nabla n|^2) \nabla^{k-l} (n+w_0) \nabla^{k+2} n \ dx\\
& =-\int \sum_{l=0}^k \sum_{m=0}^l C_k^l  C_l^m \nabla^{m+1} n \nabla^{l+1-m} n \nabla^{k-l} (n+w_0) \nabla^{k+2} n \ dx\\
& =-\int |\nabla n|^2 \nabla^k n \nabla^{k+2} n \ dx
   -\int \sum_{l=1}^{k-1} \sum_{m=0}^{l-1} C_k^l  C_l^m \nabla^{m+1} n \nabla^{l+1-m} n \nabla^{k-l}n \nabla^{k+2}n \ dx\\
&\ -\!\!\int \!\sum_{l=1}^{k-1} C_k^l \nabla^{l+1} n \nabla n \nabla^{k-l} n \nabla^{k+2} n  dx\!
   -\!\!\int\!\! \sum_{m=0}^{k} C_k^m  \nabla^{m+1} n \nabla^{k+1-m}n \ (n+w_0) \nabla^{k+2}n  dx\\
&=I_{41}+I_{42}+I_{43}+I_{44}.
\end{aligned}
\end{equation}
First of all, by Holder inequality and \eqref{GN}, it is easy to get
\begin{equation}\label{2321}
\begin{aligned}
I_{41}
&\le \|\nabla n\|_{L^6}\|\nabla n\|_{L^6}\|\nabla^k n\|_{L^6}\|\nabla^{k+2} n\|_{L^2}\\
&\lesssim \|\nabla^2 n\|_{L^2} \|\nabla n\|_{L^2}^{1-\frac{1}{k+1}}\|\nabla^{k+2} n\|_{L^2}^{\frac{1}{k+1}}\\
&\quad \times  \|\nabla n\|_{L^2}^{\frac{1}{k+1}}\|\nabla^{k+2} n\|_{L^2}^{1-\frac{1}{k+1}}
          \|\nabla^{k+2} n\|_{L^2}\\
&\lesssim \|\nabla^2 n\|_{L^2}\|\nabla n\|_{L^2}\|\nabla^{k+2} n\|_{L^2}^2\\
&\lesssim \delta \|\nabla^{k+2} n\|_{L^2}^2.
\end{aligned}
\end{equation}
For the term $I_{42}$, exploiting the Holder inequality and interpolation inequality \eqref{GN}
for the case $1\le l \le \left[\frac{k-1}{2}\right]$, we obtain directly
\begin{equation}\label{2322}
\begin{aligned}
&\|\nabla^{m+1} n\|_{L^6}\|\nabla^{l+1-m} n\|_{L^6}\|\nabla^{k-l} n\|_{L^6}\|\nabla^{k+2} n\|_{L^2}\\
&\lesssim \|\nabla^\alpha n\|_{L^2}^{1-\frac{m+1}{k}}\|\nabla^{k+2} n\|_{L^2}^{\frac{m+1}{k}}
          \|\nabla^2 n\|_{L^2}^{1-\frac{l-m}{k}}\|\nabla^{k+2} n\|_{L^2}^{\frac{l-m}{k}}\\
&\quad    \times \|\nabla^2 n\|_{L^2}^{\frac{l+1}{k}} \|\nabla^{k+2} n\|_{L^2}^{1-\frac{l+1}{k}}
          \|\nabla^{k+2} n\|_{L^2}\\
&\lesssim \|\nabla^\alpha n\|_{L^2}^{1-\frac{m+1}{k}}\|\nabla^2 n\|_{L^2}^{1+\frac{m+1}{k}}
          \|\nabla^{k+2} n\|_{L^2}^2\\
&\lesssim \delta \|\nabla^{k+2} n\|_{L^2}^2,
\end{aligned}
\end{equation}
where $\alpha$ is defined by
\begin{equation*}
\alpha=2-\frac{k}{k-(m+1)}\in (0,1).
\end{equation*}
Similarly, for the case $\left[\frac{k-1}{2}\right]+1 \le l \le k-1$, it is easy to get
\begin{equation}\label{2323}
\begin{aligned}
&\|\nabla^{m+1} n\|_{L^6}\|\nabla^{l+1-m} n\|_{L^6}\|\nabla^{k-l} n\|_{L^6}\|\nabla^{k+2} n\|_{L^2}\\
&\lesssim \|\nabla^2 n\|_{L^2}^{1-\frac{m}{k}}\|\nabla^{k+2} n\|_{L^2}^{\frac{m}{k}}
          \|\nabla^2 n\|_{L^2}^{1-\frac{l-m}{k}}\|\nabla^{k+2} n\|_{L^2}^{\frac{l-m}{k}}\\
&\quad    \times \|\nabla^\alpha n\|_{L^2}^{\frac{l}{k}} \|\nabla^{k+2} n\|_{L^2}^{1-\frac{l}{k}}
          \|\nabla^{k+2} n\|_{L^2}\\
&\lesssim \|\nabla^2 n\|_{L^2}^{2-\frac{l}{k}} \|\nabla^\alpha n\|_{L^2}^{\frac{l}{k}}
          \|\nabla^{k+2} n\|_{L^2}^2\\
&\lesssim \delta \|\nabla^{k+2} n\|_{L^2}^2,
\end{aligned}
\end{equation}
where $\alpha$ is defined by
\begin{equation*}
\alpha=2-\frac{k}{l}\in [0,1).
\end{equation*}
The combination of \eqref{2322} and \eqref{2323} yields immediately
\begin{equation}\label{2324}
I_{42} \lesssim \delta \|\nabla^{k+2} n\|_{L^2}^2.
\end{equation}
In view of \eqref{GN} and Holder inequality, the term $I_{43}$ can be estimated as
\begin{equation}\label{2325}
\begin{aligned}
I_{43}
&\le  \sum_{l=1}^{k-1} \|\nabla^{l+1} n\|_{L^6} \|\nabla n\|_{L^6}
       \|\nabla^{k-l} n\|_{L^6} \|\nabla^{k+2} n\|_{L^2}\\
&\lesssim  \sum_{l=1}^{k-1} \|\nabla n\|_{L^2}^{1-\frac{l+1}{k+1}}\|\nabla^{k+2} n\|_{L^2}^{\frac{l+1}{k+1}}
       \|\nabla^2 n\|_{L^2} \|\nabla n\|_{L^2}^{\frac{l+1}{k+1}} \|\nabla^{k+2} n\|_{L^2}^{1-\frac{l+1}{k+1}}
       \|\nabla^{k+2} n\|_{L^2}\\
&\lesssim \delta\|\nabla^{k+2} n\|_{L^2}^2.
\end{aligned}
\end{equation}
To deal with the term $I_{44}$ for the case $0 \le m \le \left[\frac{k}{2}\right]$,
by the Holder inequality and interpolation inequality \eqref{GN}, we obtain
\begin{equation}\label{2326}
\begin{aligned}
&\quad \|\nabla^{m+1} n\|_{L^3}\|\nabla^{k+1-m}n\|_{L^6}\|\nabla^{k+2}n\|_{L^2}\\
&\lesssim \|\nabla^\alpha n\|_{L^2}^{1-\frac{m}{k}}\|\nabla^{k+2} n\|_{L^2}^{\frac{m}{k}}
         \|\nabla^2 n\|_{L^2}^{\frac{m}{k}}\|\nabla^{k+2} n\|_{L^2}^{1-\frac{m}{k}}
         \|\nabla^{k+2} n\|_{L^2}\\
&\lesssim \delta \|\nabla^{k+2} n\|_{L^2}^2,
\end{aligned}
\end{equation}
where $\alpha$ is defined by
\begin{equation*}
\alpha=2-\frac{k}{2(k-m)} \in \left[1, \frac{3}{2}\right).
\end{equation*}
Similarly, for the case $\left[\frac{k}{2}\right]+1 \le m \le k$, it is easy to find
\begin{equation}\label{2327}
\begin{aligned}
&\quad \|\nabla^{m+1} n\|_{L^3}\|\nabla^{k+1-m}n\|_{L^6}\|\nabla^{k+2}n\|_{L^2}\\
&\lesssim \|\nabla^2 n\|_{L^2}^{1-\frac{m-\frac{1}{2}}{k}}\|\nabla^{k+2} n \|_{L^2}^{\frac{m-\frac{1}{2}}{k}}
         \|\nabla^\alpha n\|_{L^2}^{\frac{m-\frac{1}{2}}{k}}\|\nabla^{k+2} n\|_{L^2}^{1-\frac{m-\frac{1}{2}}{k}}
         \|\nabla^{k+2} n\|_{L^2}\\
&\lesssim \delta \|\nabla^{k+2} n\|_{L^2}^2,
\end{aligned}
\end{equation}
where $\alpha$ is defined by
\begin{equation*}
\alpha=2-\frac{k}{2m-1} \in \left[1, \frac{3}{2}\right).
\end{equation*}
Combining \eqref{2326} with \eqref{2327}, the term $I_{44}$ can be estimated as
\begin{equation}\label{2328}
I_{44}\lesssim \delta \|\nabla^{k+2} n\|_{L^2}^2.
\end{equation}
Inserting \eqref{2321}, \eqref{2324}, \eqref{2325} and \eqref{2328} into \eqref{2320}, then we obtain
\begin{equation}\label{2329}
I_{4} \lesssim \delta \|\nabla^{k+2} n\|_{L^2}^2.
\end{equation}
Substituting \eqref{2313},\eqref{2318},\eqref{2319} and \eqref{2329} into \eqref{233}, we complete the proof of lemma.
\end{proof}

\subsection{Proof of Theorem \ref{global-existence}}

\quad In this subsection, we will construct the global existence of solution for the incompressible nematic liquid crystal
flows \eqref{1.1}-\eqref{1.3}. In fact, after having the estimates \eqref{inequality1} and \eqref{inequality2} at hand, the proof for the Theorem \ref{global-existence} follows immediately.
Indeed, summing up \eqref{inequality2} from $k=0$ to $k=m(1\le m \le N)$ , then we have
\begin{equation*}
\frac{d}{dt}\left(\|u\|_{H^m}^2+\|\nabla n\|_{H^m}^2\right)
+\left(\|\nabla u\|_{H^m}^2+\|\nabla^2 n\|_{H^m}^2\right)\le 0.
\end{equation*}
which, together with \eqref{inequality1}, yields
\begin{equation}\label{2330}
\frac{d}{dt}\left(\|u\|_{H^{m}}^2+\| n\|_{H^{m+1}}^2\right)
+\left(\|\nabla u\|_{H^m}^2+\|\nabla n\|_{H^{m+1}}^2\right)\le 0.
\end{equation}
Taking $m=1$ in \eqref{2330} specially, it is easy to get
\begin{equation*}
\frac{d}{dt}\left(\|u\|_{H^1}^2+\|n\|_{H^2}^2\right)
+\left(\|\nabla u\|_{H^1}^2+\|\nabla n\|_{H^2}^2\right)\le 0.
\end{equation*}
Integrating the proceeding inequality over $[0, t]$, we find
\begin{equation*}
\|u(t)\|_{H^1}^2+\|n(t)\|_{H^2}^2 \le \|u_0\|_{H^1}^2+\|n_0\|_{H^2}^2.
\end{equation*}
Therefore, by the standard continuity argument, we can close the estimate \eqref{smallness}. Taking $m=N$ in \eqref{2330}
and integrating the resulting inequality over $[0, t]$, we complete the proof of Theorem \ref{global-existence}.

\section{Decay rates of solution}

\quad In this section, we will derive the decay rates for the velocity, pressure and direction field.
First of all, we will establish the decay rate for $d-w_0$ with $L^p$ norm by energy method developed
by Kawashima et al. \cite{K-N-N}. Secondly, applying the method by Dai \cite{{Dai2}}, we obtain the time
convergence rates for the velocity. Since the proof of time decay rates for the velocity is standard,
we only state the results and omit the proof for brevity. Furthermore, one establishes the convergence rates
for the higher-order spatial derivatives of velocity and director motivated by Lemma \ref{motivated}. Finally,
we also study the convergence rates for the pressure and mixed space-time derivatives of velocity and director.

\subsection{ Decay rate for velocity and direction field}

\quad  In this subsection, we obtain the optimal decay rate for the director by assuming that the initial
data belongs to $L^1$ Lebesgue space additionally. Now, we establish the first inequality.

\begin{lemm}\label{L1}
Under the assumption \eqref{smallness}, then we have
\begin{equation}\label{L1}
\int |n| dx \le C_0.
\end{equation}
\end{lemm}
\begin{proof}
Multiplying $\eqref{equations2}_3$ by $\frac{n}{|n|}$ and integrating over $\mathbb{R}^3$(by part), it arrives at
\begin{equation}\label{311}
\begin{aligned}
\frac{d}{dt}\int |n| dx
=\int \Delta n \frac{n}{|n|}dx
+\int |\nabla n|^2 \frac{(n+w_0)\cdot n}{|n|}dx.
\end{aligned}
\end{equation}
By virtue of integration by part and Cauchy inequality, it is easy to obtain
\begin{equation}\label{312}
\begin{aligned}
\int \Delta n \frac{n}{|n|}dx
&=-\int \frac{|\nabla n|^2}{|n|}dx
  +\int\frac{|n \cdot \nabla n|^2}{|n|^3}dx\\
&\le -\int \frac{|\nabla n|^2}{|n|}dx
  +\int\frac{|n|^2 |\nabla n|^2}{|n|^3}dx=0.
\end{aligned}
\end{equation}
In view of $|(n+w_0)(x,t)|=1$, plugging \eqref{312} into \eqref{311}, we get immediately
\begin{equation*}
\begin{aligned}
\frac{d}{dt}\int |n| dx
\le \int |\nabla n|^2 dx,
\end{aligned}
\end{equation*}
which, integrating over $[0, t]$, yields
\begin{equation*}
\int |n| dx \le \int |n_0|dx+\int_0^t \int |\nabla n|^2 dx d\tau,
\end{equation*}
which, together with \eqref{1.5}, completes the proof of the lemma.
\end{proof}

Next, we will establish the $L^p$ integrability for the director.

\begin{lemm}\label{Lp}
Under the assumption \eqref{smallness}, then we have for any $p \ge 2$,
\begin{equation}\label{Lp}
\frac{d}{dt}\int |n|^p dx+C_p \int |\nabla |n|^{\frac{p}{2}}|^2 dx \le 0,
\end{equation}
where $C_p$ is constant depending on $p$.
\end{lemm}
\begin{proof}
Multiplying $\eqref{equations2}_3$ by $|n|^{p-2}n$ and integrating over $\mathbb{R}^3$(by part), we obtain
\begin{equation}\label{321}
\frac{1}{p}\frac{d}{dt}\int |n|^p dx-\!\int \Delta n|n|^{p-2} n dx
=\!\int |\nabla n|^2 |n|^{p-2}(n+w_0)\cdot n dx.
\end{equation}
The integration by part yields at once
\begin{equation}\label{322}
-\int \Delta n|n|^{p-2}n dx=(p-2)\int |n\cdot \nabla n|^{2}|n|^{p-4}dx+\int |\nabla n|^2 |n|^{p-2}dx.
\end{equation}
With the help of $|(n+w_0)(x,t)|=1$ and Sobolev inequality, it arrives at
\begin{equation}\label{323}
\int |\nabla n|^2 |n|^{p-2}(n+w_0)\cdot n dx
\le \int |\nabla n|^2 |n|^{p-1}dx
\lesssim \delta \int |\nabla n|^2 |n|^{p-2}dx.
\end{equation}
Substituting \eqref{322} and \eqref{323} into \eqref{321}, by virtue of smallness of  $\delta $, then we obtain
\begin{equation}\label{324}
\frac{1}{p}\frac{d}{dt}\int |n|^p dx
+\left(p-\frac{3}{2}\right)\int |n \cdot \nabla n|^2 |n|^{p-4}dx\le 0.
\end{equation}
Computing directly will give the following identity, i.e.,
\begin{equation}\label{325}
\int |n \cdot \nabla n|^2 |n|^{p-4}dx
=\frac{4}{p^2}\int |\nabla |n|^\frac{p}{2}|^2 dx.
\end{equation}
Inserting  \eqref{325} into \eqref{324}, it follows
\begin{equation*}
\frac{d}{dt}\int |n|^p dx
+\frac{4\left(p-\frac{3}{2}\right)}{p}\int |\nabla |n|^\frac{p}{2}|^2 dx\le 0.
\end{equation*}
Therefore, we  complete the proof of the lemma.
\end{proof}

Finally, we will establish the decay rates for the director as follows.

\begin{lemm}\label{director-Decay-Lp}
Under the assumption \eqref{smallness}, then we have for any $p \ge 2$
\begin{equation}\label{director-Decay}
\|n\|_{L^p} \le C(1+t)^{-\frac{3}{2}\left(1-\frac{1}{p}\right)}
\end{equation}
\end{lemm}
\begin{proof}
Multiplying \eqref{Lp} by $(1+t)^{\alpha}$($\alpha$ is to defined below),
then we find
\begin{equation}\label{331}
\frac{d}{dt}\left[(1+t)^\alpha \int |n|^p dx\right]+C_p(1+t)^{\alpha}\int |\nabla |n|^\frac{p}{2}|^2 dx
\le \alpha(1+t)^{\alpha-1} \int |n|^p dx.
\end{equation}
By virtue of Sobolev and Young inequalities, it follows that
\begin{equation}\label{332}
\begin{aligned}
&\alpha(1+t)^{\alpha-1} \int |n|^p dx\\
&\le C(\alpha) (1+t)^{\alpha-1}\|n\|_{L^1}^{\frac{2p}{3p-1}}
              \|\nabla |n|^{\frac{p}{2}}\|_{L^2}^{\frac{6(p-1)}{3p-1}}\\
& =C(\alpha) (1+t)^{\frac{3(p-1)\alpha}{3p-1}} \|\nabla |n|^{\frac{p}{2}}\|_{L^2}^{\frac{6(p-1)}{3p-1}}
              (1+t)^{\frac{2\alpha}{3p-1}-1}\|n\|_{L^1}^{\frac{2p}{3p-1}}\\
&\le \varepsilon(1+t)^\alpha \|\nabla |n|^{\frac{p}{2}}\|_{L^2}^2
     +C(\varepsilon, \alpha)(1+t)^{\alpha-\frac{3p-1}{2}}\|n\|_{L^1}^p.
\end{aligned}
\end{equation}
Plugging \eqref{332} into \eqref{331} and choosing $\varepsilon$ small enough, we have
\begin{equation}\label{333}
\begin{aligned}
\frac{d}{dt}\left[(1+t)^\alpha \int |n|^p dx\right]+\frac{C_p}{2}(1+t)^{\alpha}\int |\nabla |n|^\frac{p}{2}|^2 dx
\le C(\alpha)(1+t)^{\alpha-\frac{3p-1}{2}},
\end{aligned}
\end{equation}
where we have used \eqref{L1}. Assuming $\alpha \neq \frac{3}{2}\left(p-1\right)$ and integrating
\eqref{333} over $[0, t]$, then we obtain
\begin{equation*}
\int |n|^p dx \le C(1+t)^{-\alpha} \int |n_0|^p dx+C(1+t)^{-\frac{3}{2}(p-1)}.
\end{equation*}
Choosing $\alpha$ as a number which is larger than $\frac{3}{2}(p-1)$(for example, choosing $\alpha=1+\frac{3}{2}(p-1)$),
then we deduce directly
\begin{equation*}
\|d-w_0\|_{L^p} \le C(1+t)^{-\frac{3}{2}\left(1-\frac{1}{p}\right)}.
\end{equation*}
Therefore, we complete the proof of the lemma.
\end{proof}
\begin{rema}
In fact, the application of interpolation inequality \eqref{GN} yields
\begin{equation*}
\begin{aligned}
\|d-w_0\|_{L^p}
&\le \|d-w_0\|_{L^1}^{\frac{2}{p}-1}\|d-w_0\|_{L^2}^{2\left(1-\frac{1}{p}\right)}\\
&\le C\|d-w_0\|_{L^1}^{\frac{2}{p}-1}(1+t)^{-\frac{3}{2}\left(1-\frac{1}{p}\right)}\\
&\le C(1+t)^{-\frac{3}{2}\left(1-\frac{1}{p}\right)},
\end{aligned}
\end{equation*}
for the case $1\le p <2.$
\end{rema}

After having the decay rates of director at hand, the convergence rates of velocity is easy to establish.
Just following the idea by Dai et al. \cite{Dai2}, we have the following lemma.

\begin{lemm}\label{velocity-Decay}
Under the assumption \eqref{smallness}, the velocity has the following time decay rate
\begin{equation}\label{velocity-Decay}
\|u\|_{L^2} \le C(1+t)^{-\frac{3}{4}}.
\end{equation}
\end{lemm}

\subsection{Decay rates for higher-order spatial derivatives of velocity and direction field}

\quad In this subsection, we will establish the time convergence rates for the higher-order spatial derivatives
of velocity and director. First of all, we have the following time decay rates for the first-order spatial
derivatives of director.
\begin{lemm}\label{director-decay2}
Under the assumptions of Theorem \ref{rates1}, then the director has the following time decay rate
\begin{equation}\label{director-decay2}
\|\nabla n\|_{L^2} \le C (1+t)^{-\frac{5}{4}}.
\end{equation}
\end{lemm}
\begin{proof}
Taking spatial derivatives to $\eqref{equations2}_3$, multiplying the resulting identity by $\nabla n$
and integrating over $\mathbb{R}^3$, then we have
\begin{equation}\label{351}
\frac{1}{2}\frac{d}{dt}\int |\nabla n|^2 dx+\int |\nabla^2 n|^2 dx
=\int \nabla (-u \cdot \nabla n+|\nabla n|^2(n+w_0)) \nabla ndx.
\end{equation}
Integrating by parts and applying the Holder and Sobolev inequalities, we have
\begin{equation}\label{352}
\int u \cdot \nabla n \nabla^2 ndx
\le \|u\|_{L^3}\|\nabla n\|_{L^6}\|\nabla^2 n\|_{L^2}
\lesssim \delta \|\nabla^2 n\|_{L^2}^2.
\end{equation}
On the other hand, we have
\begin{equation}\label{353}
-\int |\nabla n|^2(n+w_0)\nabla^2 n dx
\le \|\nabla n\|_{L^3}\|\nabla n\|_{L^6}\|\nabla^2 n\|_{L^2}
\lesssim \delta \|\nabla^2 n\|_{L^2}^2.
\end{equation}
By virtue of the smallness of $\delta$, substituting \eqref{352} and \eqref{353} into \eqref{351}, we have
\begin{equation}\label{354}
\frac{d}{dt}\int |\nabla n|^2 dx+\int |\nabla^2 n|^2 dx\le 0.
\end{equation}
Denoting the time sphere $S_0(t)$ $($see \cite{Schonbek}$)$ as
\begin{equation*}
S_0(t)=\left\{\xi \in \mathbb{R}^3 \left||\xi|\le \left(\frac{R}{1+t}\right)^{\frac{1}{2}}\right.\right\},
\end{equation*}
where $R$ is a constant defined below. By virtue of Parseval identity , it arrives at
\begin{equation*}
\begin{aligned}
\int |\nabla^2 n|^2 dx
&\ge \int_{\mathbb{R}^3 \backslash S_0(t)} |\xi|^{4}|\hat{n}|^2 dx\\
&\ge \frac{R}{1+t}\int_{\mathbb{R}^3 \backslash S_0(t)} |\xi|^{2}|\hat{n}|^2 d\xi\\
&=\frac{R}{1+t}\int_{\mathbb{R}^3 } |\xi|^{2}|\hat{n}|^2 d\xi
  -\frac{R}{1+t}\int_{S_0(t)} |\xi|^{2}|\hat{n}|^2 d\xi\\
&\ge \frac{R}{1+t}\int_{\mathbb{R}^3 } |\xi|^{2}|\hat{n}|^2 d\xi
  -\frac{R^2}{(1+t)^2}\int_{\mathbb{R}^3} |\hat{n}|^2 d\xi.\\
\end{aligned}
\end{equation*}
Hence, we have
\begin{equation}\label{FSM}
\int |\nabla^2 n|^2 dx \ge \frac{R}{1+t}\int_{\mathbb{R}^3} |\nabla n|^2 dx
  -\frac{R^2}{(1+t)^2}\int_{\mathbb{R}^3} |n|^2 dx.
\end{equation}
Combining \eqref{354} and \eqref{FSM}, it follows
\begin{equation*}
\frac{d}{dt}\int |\nabla n|^2 dx+\frac{R}{1+t}\int_{\mathbb{R}^3 } |\nabla n|^2 dx
\le \frac{R^2}{(1+t)^2}\int_{\mathbb{R}^3} |n|^2 dx,
\end{equation*}
or equivalently, we have
\begin{equation}\label{356}
\frac{d}{dt}\int |\nabla n|^2 dx+\frac{R}{1+t}\int_{\mathbb{R}^3 } |\nabla n|^2 dx
\lesssim (1+t)^{-\frac{7}{2}}.
\end{equation}
Choosing $R=3$ and multiplying \eqref{356} by $(1+t)^3$, we get
\begin{equation}\label{357}
\frac{d}{dt}\left[(1+t)^3\|\nabla n\|_{L^2}^2\right]\le C(1+t)^{-\frac{1}{2}}.
\end{equation}
Integrating \eqref{357} over $[0, t]$, then we have
\begin{equation*}
\|\nabla n\|_{L^2}^2 \le C(1+t)^{-\frac{5}{2}}.
\end{equation*}
Therefore, we complete the proof of the lemma.
\end{proof}

\begin{rema}
Combining the results \eqref{director-Decay} and \eqref{director-decay2}, it follows immediately
\begin{equation}\label{velocity-director-decay1}
\|u\|_{L^2}^2+\|\nabla n\|_{L^2}^2\le C(1+t)^{-\frac{3}{2}}.
\end{equation}
\end{rema}

Next, we will establish the time decay rates for the higher-order spatial derivatives of velocity
and director in the following lemma.

\begin{lemm}\label{velocity-rates}
Under the assumptions of Theorem \ref{rates1}, the global solution $(u, n)$ of problem \eqref{equations2}-\eqref{infinity-decay2} satisfies
\begin{equation}\label{velocity-director-decay2}
\|\nabla^k u\|_{L^2}^2+\|\nabla^{k+1} n\|_{L^2}^2 \le C (1+t)^{-\frac{3}{2}-k},
\end{equation}
for $k=0,1,...,N$.
\end{lemm}
\begin{proof}
We will give the proof for \eqref{velocity-director-decay2} by taking the strategy of induction.
In fact, the inequality \eqref{velocity-director-decay1} implies the equality
\eqref{velocity-director-decay2} holds on for the case $k=0$. Now, we suppose the inequality
\eqref{velocity-director-decay2}  holds on for the case $k=l$, i.e.,
\begin{equation}\label{361}
\|\nabla^l u\|_{L^2}^2+\|\nabla^{l+1} n\|_{L^2}^2 \le C (1+t)^{-\frac{3}{2}-l},
\end{equation}
where $l=0, 1, ..., N-1$.
In the sequel, we focus on verifying \eqref{velocity-director-decay2} holds on for the case $k=l+1$.
In fact, taking $k=l+1$ in the inequality \eqref{inequality2}, it follows immediately
\begin{equation}\label{362}
\frac{d}{dt}\int (|\nabla^{l+1} u|^2+|\nabla^{l+2} n|^2)dx+\int(|\nabla^{l+2} u|^2+|\nabla^{l+3} n|^2) dx\le 0.
\end{equation}
Similar to \eqref{FSM}, we have
\begin{equation}\label{363}
\int |\nabla^{l+2} u|^2 dx \ge \frac{R}{1+t}\int_{\mathbb{R}^3} |\nabla^{l+1} u|^2 dx
  -\frac{R^2}{(1+t)^2}\int_{\mathbb{R}^3} |\nabla^l u|^2 dx.
\end{equation}
and
\begin{equation}\label{364}
\int |\nabla^{l+3} n|^2 dx \ge \frac{R}{1+t}\int_{\mathbb{R}^3} |\nabla^{l+2} n|^2 dx
  -\frac{R^2}{(1+t)^2}\int_{\mathbb{R}^3} |\nabla^{l+1} n|^2 dx.
\end{equation}
Inserting \eqref{363} and \eqref{364} into \eqref{362} and applying the convergence rates \eqref{361}, we get
\begin{equation}\label{365}
\begin{aligned}
&\frac{d}{dt}\int (|\nabla^{l+1} u|^2+|\nabla^{l+2} n|^2) dx
+\frac{R}{1+t}\int (|\nabla^{l+1} u|^2+|\nabla^{l+2} n|^2) dx\\
&\lesssim \frac{R^2}{(1+t)^2}\int (|\nabla^l u|^2+|\nabla^{l+1} n|^2) dx\\
&\lesssim (1+t)^{-\frac{7}{2}-l}.
\end{aligned}
\end{equation}
Choosing $R=l+3$ and multiplying \eqref{365} by $(1+t)^{l+3}$, it arrives at
\begin{equation}\label{366}
\frac{d}{dt}\left[(1+t)^{l+3}\|\nabla^{l+1}(u, \nabla n)\|_{L^2}^2\right]\le C(1+t)^{-\frac{1}{2}}.
\end{equation}
Integrating \eqref{366} over $[0, t]$, then we get
\begin{equation*}
\|\nabla^{l+1}(u, \nabla n)\|_{L^2}^2 \le C(1+t)^{-\frac{5}{2}-l}.
\end{equation*}
Therefore, by the general step of induction, we complete the proof of lemma.
\end{proof}

Next, we will enhance the time convergence rates for the direction field. This improvement is motivated
by the following lemma.
\begin{lemm}\label{motivated}
For some smooth function $F(x,t)$, suppose the function $v(x,t)$ is the solution of heat equation
\begin{equation}\label{motivated-equation}
v_t-\Delta v=F,
\end{equation}
for $(x,t)\in \mathbb{R}^3 \times R^+$ with smooth initial data $v(x,0)=v_0$. If the function $F$ and
the solution $v$ have the time decay rates
\begin{equation}\label{motivated-decay}
\|\nabla^k v\|_{L^2}^2 \le C(1+t)^{-(\frac{3}{2}+k)}, \quad  \|\nabla^k F\|_{L^2}^2 \le C(1+t)^{-\alpha}
\end{equation}
where $\alpha\ge k+\frac{7}{2}$. Then, we have the following decay rate
\begin{equation*}
\|\nabla^{k+1} v(t)\|_{L^2}^2 \le C(1+t)^{-(k+\frac{5}{2})}.
\end{equation*}
\end{lemm}
\begin{proof}
Taking $(k+1)-$th spatial derivative on both hand sides of \eqref{motivated-equation},
multiplying by $\nabla^{k+1} v$ and integrating over $\mathbb{R}^3$, we obtain
\begin{equation*}
\frac{1}{2}\frac{d}{dt}\int |\nabla^{k+1} v|^2 dx+\int |\nabla^{k+2}v|^2 dx
\le \frac{1}{2}\int |\nabla^{k} F|^2 dx +\frac{1}{2}\int |\nabla^{k+2} v|^2 dx,
\end{equation*}
which implies
\begin{equation}\label{371}
\frac{d}{dt}\int |\nabla^{k+1} v|^2 dx+\int |\nabla^{k+2}v|^2 dx\le \int |\nabla^{k} F|^2 dx.
\end{equation}
Similar to \eqref{FSM}, it follows immediately
\begin{equation*}
\int |\nabla^{k+2}v|^2 dx
\ge \frac{k+3}{1+t}\int |\nabla^{k+1}v|^2 dx-\left(\frac{k+3}{1+t}\right)^2 \int |\nabla^{k}v|^2 dx,
\end{equation*}
which, together with \eqref{motivated-decay} and  \eqref{371}, yields
\begin{equation}\label{372}
\begin{aligned}
&\frac{d}{dt}\int |\nabla^{k+1} v|^2 dx+\frac{k+3}{1+t}\int |\nabla^{k+1}v|^2 dx\\
&\le \left(\frac{k+3}{1+t}\right)^2 \int |\nabla^{k}v|^2 dx+\int |\nabla^{k} F|^2 dx\\
&\le C(1+t)^{-(k+\frac{7}{2})}.
\end{aligned}
\end{equation}
Multiplying \eqref{372} by $(1+t)^{k+3}$ and integrating over $[0, t]$, we get
\begin{equation*}
\|\nabla^{k+1} v(t)\|_{L^2}^2 \le (1+t)^{-k-3}\left[\|\nabla^{k+1} v_0\|_{L^2}^2+C(1+t)^{\frac{1}{2}}\right],
\end{equation*}
or equivalently
\begin{equation*}
\|\nabla^{k+1} v(t)\|_{L^2}^2 \le C(1+t)^{-(k+\frac{5}{2})}.
\end{equation*}
Therefore, we complete the proof of lemma.
\end{proof}

\begin{lemm}\label{director-rates}
Under the assumptions of Theorem \ref{rates1}, the director has the following time decay rates
\begin{equation}\label{director-decay3}
\|\nabla^{k} n\|_{L^2}^2 \le C (1+t)^{-\frac{3}{2}-k},
\end{equation}
where $k=1,....,N+1$.
\end{lemm}
\begin{proof}
We will give the proof by taking the strategy of induction.
In fact, the inequality \eqref{director-decay2} implies the equality
\eqref{director-decay3} holds on for the case $k=1$. Now, we suppose the inequality
\eqref{director-decay3}  holds on for the case $k=l(l=1,...,N)$, i.e.,
\begin{equation}\label{381}
\|\nabla^{l} n\|_{L^2}^2 \le C (1+t)^{-\frac{3}{2}-l}.
\end{equation}
In the sequel, we focus on verifying \eqref{director-decay3} holds on for the case $k=l+1$.
For the case $N=1$, taking $k=1$ in \eqref{232}, then we have
\begin{equation}\label{382}
\!\frac{1}{2}\!\frac{d}{dt}\!\int\! |\nabla^2 n|^2 dx\!+\int |\nabla^3 n|^2 dx
=\int \!\nabla^2 \!\left[-u \cdot \nabla n\!+|\nabla n|^2(n+w_0)\right]\! \cdot \nabla^2 n \ dx
=I\!I_1+I\!I_2.
\end{equation}
It follows from integration by part, Holder inequality and \eqref{GN} that
\begin{equation}\label{383}
\begin{aligned}
I\!I_1
&\lesssim (\|\nabla n\|_{L^\infty}\|\nabla u\|_{L^2}+\|u\|_{L^3}\|\nabla^2 n\|_{L^6})\|\nabla^3 n\|_{L^2}\\
&\lesssim (\|\nabla^2 n\|_{L^2}^{\frac{1}{2}}\|\nabla^3 n\|_{L^2}^{\frac{1}{2}}\|\nabla u\|_{L^2}
           +\|u\|_{H^1}\|\nabla^3 n\|_{L^2})\|\nabla^3 n\|_{L^2}\\
&\lesssim \|\nabla^2 n\|_{L^2}^2\|\nabla u\|_{L^2}^2+\delta\|\nabla^3 n\|_{L^2}^2\\
&\lesssim (1+t)^{-5}+\delta\|\nabla^3 n\|_{L^2}^2.
\end{aligned}
\end{equation}
On the other hand, by virtue of \eqref{2329}, it arrives at immediately
\begin{equation}\label{384}
I\!I_2 \lesssim \delta  \|\nabla^3 n\|_{L^2}^2.
\end{equation}
Substituting \eqref{383} and \eqref{384} into \eqref{382}, we get
\begin{equation}\label{385}
\frac{d}{dt}\int |\nabla^2 n|^2 dx+\int |\nabla^3 n|^2 dx \lesssim (1+t)^{-5}.
\end{equation}
For the case $N \ge2$, then taking $k=l$ in \eqref{232}, then the first term on the right hand side of \eqref{232} can be estimated as follows
\begin{equation}\label{386}
\begin{aligned}
&-\int \nabla^{l+1}(u \cdot \nabla n)\nabla^{l+1}n dx\\
&\lesssim \sum_{m=1}^{l-1}\|\nabla^m u\|_{L^3}\|\nabla^{l+1-m}n\|_{L^6}\|\nabla^{l+2}n\|_{L^2}
          +\|u\|_{L^3}\|\nabla^{l+1}n\|_{L^6}\|\nabla^{l+2}n\|_{L^2}\\
&\quad \  +\|\nabla n\|_{L^\infty}\|\nabla^{l}u\|_{L^2}\|\nabla^{l+2}n\|_{L^2}\\
&\lesssim \sum_{m=1}^{l-1}\|\nabla^m u\|_{H^1}^2\|\nabla^{l+2-m}n\|_{L^2}^2
          +\|\nabla n\|_{H^2}^2\|\nabla^{l}u\|_{L^2}^2+(\varepsilon+\delta)\|\nabla^{l+2}n\|_{L^2}^2\\
&\lesssim \sum_{m=1}^{l-1}(1+t)^{-\frac{3}{2}-m}(1+t)^{-\frac{7}{2}-l+m}+(1+t)^{-4-l}
          +(\varepsilon+\delta)\|\nabla^{l+2}n\|_{L^2}^2\\
&\lesssim (1+t)^{-4-l} +(\varepsilon+\delta)\|\nabla^{l+2}n\|_{L^2}^2.\\
\end{aligned}
\end{equation}
On the other hand, by virtue of \eqref{2329}, it follows immediately
\begin{equation}\label{387}
\int \nabla^{l+1} (|\nabla n|^2 (n+w_0)) \cdot \nabla^{l+1} n \ dx \lesssim \delta \|\nabla^{l+2} n\|_{L^2}^2.
\end{equation}
Inserting \eqref{386} and \eqref{387} into \eqref{2329}, we get
\begin{equation}\label{388}
\frac{d}{dt}\int |\nabla^{l+1} n|^2 dx+\int |\nabla^{l+2} n|^2 dx\le C(1+t)^{-4-l}.
\end{equation}
which, together with \eqref{385}, implies \eqref{388} holds on for any $N \ge 1.$
Similar to \eqref{FSM}, we have
\begin{equation}\label{389}
\int |\nabla^{l+2} n|^2 dx \ge \frac{R}{1+t}\int_{\mathbb{R}^3} |\nabla^{l+1} n|^2 dx
-\frac{R^2}{(1+t)^2}\int_{\mathbb{R}^3} |\nabla^l n|^2 dx.
\end{equation}
Combining \eqref{388} and \eqref{389} and applying the decay rates \eqref{381}, it follows
\begin{equation}\label{3810}
\begin{aligned}
&\frac{d}{dt}\int |\nabla^{l+1} n|^2 dx+\frac{R}{1+t}\int_{\mathbb{R}^3} |\nabla^{l+1} n|^2 dx\\
&\le \frac{R^2}{(1+t)^2}\int_{\mathbb{R}^3} |\nabla^l n|^2 dx+C(1+t)^{-4-l}\\
&\le C(1+t)^{-\frac{7}{2}-l}.
\end{aligned}
\end{equation}
Choosing $R=l+3$ and multiplying \eqref{3810} by $(1+t)^{l+3}$, then we have
\begin{equation}\label{3811}
\frac{d}{dt}\left[(1+t)^{l+3}\|\nabla^{l+1} n\|_{L^2}^2\right]\le C(1+t)^{-\frac{1}{2}}.
\end{equation}
Integrating \eqref{3811} over $[0, t]$, it arrives at
\begin{equation*}
\|\nabla^{l+1} n\|_{L^2}^2 \le C(1+t)^{-\frac{5}{2}-l}.
\end{equation*}
Therefore, we complete the proof of lemma.
\end{proof}

\emph{\bf{Proof for Theorem \ref{rates1}:}}\ With the help of the Lemma \ref{director-Decay-Lp}, Lemma \ref{velocity-rates}
 and Lemma \ref{director-rates}, we complete the proof of Theorem \ref{rates1}.

\subsection{Decay rates for  pressure and space-time derivatives of velocity and direction field}

\quad In this section, we will establish the convergence rates for the pressure and the mixed space-time
derivatives of velocity and director.
\begin{lemm}\label{director-mixed}
Under the assumptions of Theorem \ref{rates1}, the director has the following time decay rates
\begin{equation*}
\|\nabla^{k} n_t(t)\|_{L^2}^2 \le C (1+t)^{-\frac{7}{2}-k},
\end{equation*}
for $k=0,1,2,....,N-1$.
\end{lemm}
\begin{proof}
Taking $k-$th spatial derivatives on both hand side of $\eqref{equations2}_3$, multiplying by $\nabla^k n_t$
and integrating over $\mathbb{R}^3$, then we have
\begin{equation}\label{391}
\|\nabla^k n_t\|_{L^2}^2
=\int \nabla^k (-u \cdot \nabla u+\Delta n+|\nabla n|^2(n+w_0))\nabla^k n_t dx.
\end{equation}
For the case $N=1$, then it follows immediately
\begin{equation}\label{392}
\begin{aligned}
\|n_t\|_{L^2}^2
&\le (\|\Delta n\|_{L^2}+\|u\|_{L^6}\|\nabla n\|_{L^3}+\|\nabla n\|_{L^3}\|\nabla n\|_{L^6})\|n_t\|_{L^2}\\
&\lesssim \|\Delta n\|_{L^2}^2+\|\nabla u\|_{L^2}^2\|\nabla n\|_{H^1}^2
          +\|\nabla n\|_{H^1}\|\nabla^2 n\|_{L^2}^2+\varepsilon\|n_t\|_{L^2}^2\\
&\lesssim (1+t)^{-\frac{7}{2}}+\varepsilon\|n_t\|_{L^2}^2.
\end{aligned}
\end{equation}
For the case $N\ge2,$ it follows from the Young inequality that
\begin{equation}\label{393}
\int \nabla^k \Delta n \nabla^k n_t
\lesssim \|\nabla^{k+2}n\|_{L^2}^2+\varepsilon\|\nabla^k n_t\|_{L^2}^2
\lesssim (1+t)^{-\frac{7}{2}-k}+\varepsilon\|\nabla^k n_t\|_{L^2}^2.
\end{equation}
By virtue of \eqref{1.7}, Holder and Sobolev inequalities, we get
\begin{equation}\label{394}
\begin{aligned}
&-\int \nabla^k(u\cdot \nabla n)\nabla^k n_t dx\\
&\lesssim \sum_{l=0}^k \|\nabla^l u\|_{H^1}^2\|\nabla^{k+2-l}n\|_{L^2}^2
          +\varepsilon\|\nabla^k n_t\|_{L^2}^2\\
&\lesssim \sum_{l=0}^k(1+t)^{-\frac{3}{2}-l}(1+t)^{-\frac{7}{2}-k+l}+\varepsilon\|\nabla^k n_t\|_{L^2}^2\\
&\lesssim (1+t)^{-5-k}+\varepsilon\|\nabla^k n_t\|_{L^2}^2.\\
\end{aligned}
\end{equation}
Similarly, we have
\begin{equation}\label{395}
\begin{aligned}
&\int \nabla^{k}(|\nabla n|^2(n+w_0))\nabla^k n_t dx\\
&\lesssim \sum_{l=0}^{k-1}\sum_{m=0}^l \|\nabla^{m+1}n\|_{L^6}\|\nabla^{l+1-m}n\|_{L^6}\|\nabla^{k-l}n\|_{L^6}
          \|\nabla^k n_t\|_{L^2}\\
&\quad  +\sum_{m=0}^{k} \|\nabla^{m+1}n\|_{L^3}\|\nabla^{k+1-m}n\|_{L^6}\|\nabla^k n_t\|_{L^2}\\
&\lesssim \sum_{l=0}^{k-1}\sum_{m=0}^l \|\nabla^{m+2}n\|_{L^2}^2\|\nabla^{l+2-m}n\|_{L^2}^2\|\nabla^{k+1-l}n\|_{L^2}^2\\
&\quad  +\sum_{m=0}^{k} \|\nabla^{m+1}n\|_{H^1}^2\|\nabla^{k+2-m}n\|_{L^2}^2+\varepsilon\|\nabla^k n_t\|_{L^2}^2\\
&\lesssim \sum_{l=0}^{k-1}\sum_{m=0}^l (1+t)^{-\frac{7}{2}-m}(1+t)^{-\frac{7}{2}-l+m}(1+t)^{-\frac{5}{2}-k+l}\\
&\quad  +\sum_{m=0}^{k}(1+t)^{-\frac{5}{2}-m}(1+t)^{-\frac{7}{2}-k+m}+\varepsilon\|\nabla^k n_t\|_{L^2}^2\\
&\lesssim (1+t)^{-6-k}+\varepsilon\|\nabla^k n_t\|_{L^2}^2.
\end{aligned}
\end{equation}
Substituting \eqref{393}, \eqref{394} and \eqref{395} into \eqref{391},then we have
\begin{equation}\label{396}
\|\nabla^k n_t\|_{L^2}^2 \lesssim (1+t)^{-\frac{7}{2}-k}+\varepsilon\|\nabla^k n_t\|_{L^2}^2.
\end{equation}
Combination of \eqref{392} and \eqref{396} and the smallness of $\varepsilon$, we complete the proof of lemma.
\end{proof}

Next, we will establish the time decay rates for the pressure and the mixed space-time derivatives of the velocity.
More precisely, we have
\begin{lemm}\label{velocity-mixed}
Under the assumptions of Theorem \ref{rates1}, the pressure and the velocity have the following
time decay rates
\begin{equation*}
\|\nabla^{k+1}P\|_{L^2}^2+\|\nabla^{k} u_t(t)\|_{L^2}^2 \le C (1+t)^{-\frac{7}{2}-k},
\end{equation*}
for $k=0,1,2,....,N-2$.
\end{lemm}
\begin{proof}
Taking $k-th$ spatial derivatives on both hand side of $\eqref{equations2}_1$, multiplying by $\nabla^k u_t$
and integrating over $\mathbb{R}^3$, then we have
\begin{equation}\label{3101}
\int |\nabla^k u_t|^2 dx
=-\int \nabla^k(u\cdot \nabla u+\Delta u+\nabla P+\nabla n \Delta n)\nabla^k u_t dx.
\end{equation}
First of all, by virtue of \eqref{1.6}, Holder and Sobolev inequalities, we obtain
\begin{equation}\label{3102}
\begin{aligned}
&-\int \nabla^k(u\nabla u)\nabla^k u_t dx\\
&\lesssim \sum_{l=0}^k \|\nabla^l u\|_{L^3}\|\nabla^{k+1-l}u\|_{L^6}\|\nabla^k u_t\|_{L^2}\\
&\lesssim \sum_{l=0}^k \|\nabla^l u\|_{H^1}^2\|\nabla^{k+2-l}u\|_{L^2}^2+\varepsilon\|\nabla^k u_t\|_{L^2}^2\\
&\lesssim \sum_{l=0}^k (1+t)^{-\frac{3}{2}-l}(1+t)^{-\frac{7}{2}-k+l}+\varepsilon\|\nabla^k u_t\|_{L^2}^2\\
&\lesssim (1+t)^{-5-k}+\varepsilon\|\nabla^k u_t\|_{L^2}^2.\\
\end{aligned}
\end{equation}
Similarly, we have
\begin{equation}\label{3103}
\int \nabla^k \Delta u \nabla^k u_t
\lesssim \|\nabla^{k+2}u\|_{L^2}^2+\varepsilon\|\nabla^k u_t\|_{L^2}^2
\lesssim (1+t)^{-\frac{7}{2}-k}+\varepsilon\|\nabla^k u_t\|_{L^2}^2.
\end{equation}
and
\begin{equation}\label{3104}
\begin{aligned}
&-\int \nabla^k(\nabla n\Delta n)\nabla^k u_tdx\\
&\lesssim \sum_{l=0}^k \|\nabla^{l+1}n\|_{L^3}\|\nabla^{k+2-l}n\|_{L^6}\|\nabla^k u_t\|_{L^2}\\
&\lesssim \sum_{l=0}^k \|\nabla^{l+1}n\|_{H^1}^2\|\nabla^{k+3-l}n\|_{L^2}^2+\varepsilon\|\nabla^k u_t\|_{L^2}^2\\
&\lesssim (1+t)^{-\frac{5}{2}-l}(1+t)^{-\frac{9}{2}-k+l}+\varepsilon\|\nabla^k u_t\|_{L^2}^2\\
&\lesssim (1+t)^{-7-k}+\varepsilon\|\nabla^k u_t\|_{L^2}^2.\\
\end{aligned}
\end{equation}
Integrating by part and applying $\eqref{equations2}_2$, it arrives at
\begin{equation}\label{3105}
-\int \nabla^{k+1}P\nabla^k u_t dx=0.
\end{equation}
Substituting \eqref{3102}-\eqref{3105} into \eqref{3101} and choosing $\varepsilon$ small enough, we get
\begin{equation}\label{3106}
\|\nabla^k u_t\|_{L^2}^2 \le C(1+t)^{-\frac{7}{2}-k}.
\end{equation}
It follows from the classical regularity of Stokes equations that
\begin{equation}\label{3107}
\|\nabla^{k+2} u\|_{L^2}^2+\|\nabla^{k+1} P\|_{L^2}^2 \lesssim \|\nabla^k (-u_t-u\cdot \nabla u-\nabla n \Delta n)\|_{L^2}^2.
\end{equation}
On the other hand, by virtue of \eqref{3102}, \eqref{3104} and \eqref{3106}, we have
\begin{equation}\label{3108}
\|\nabla^k (-u_t-u\cdot \nabla u-\nabla n \Delta n)\|_{L^2}^2 \lesssim (1+t)^{-\frac{7}{2}-k}.
\end{equation}
Plugging \eqref{3108} into \eqref{3107} gives
\begin{equation*}
\|\nabla^{k+1} P\|_{L^2}^2  \lesssim (1+t)^{-\frac{7}{2}-k},
\end{equation*}
which, together with \eqref{3106}, completes the proof of the lemma.
\end{proof}

\emph{\bf{Proof for Theorem \ref{rates2}:}}\ With the help of the Lemma \ref{director-mixed}
and Lemma \ref{velocity-mixed}, we complete the proof of Theorem \ref{rates2}.

\section*{Acknowledgements}
Qiang Tao's research was supported by the NSF(Grant No.11171060) and the NSF(Grant No.11301345).
Zheng-an Yao's research was supported in part by NNSFC(Grant No.11271381) and
China 973 Program(Grant No. 2011CB808002).

\phantomsection

\end{document}